\documentclass[12pt]{amsart}
\usepackage{amsfonts, amssymb}
\usepackage{parskip, fullpage, verbatim}
\usepackage[colorlinks=true]{hyperref}
\usepackage{breakurl}

\newcommand\CA{{\mathcal A}}

\newcommand\BBC{{\mathbb C}}
\newcommand\BBK{{\mathbb K}}

\newcommand\BBZ{{\mathbb Z}}

\newcommand {\GAP}{\textsf{GAP}}  
\newcommand {\CHEVIE}{\textsf{CHEVIE}}  
\newcommand {\Singular}{\textsf{SINGULAR}}

\newcommand\Der{{\operatorname{Der}}}

\newcommand\Fix{{\operatorname{Fix}}}

\newcommand\pdeg{\operatorname{pdeg}}

\numberwithin{equation}{section}

\theoremstyle{plain}

\newtheorem{theorem}[equation]{Theorem}

\theoremstyle{definition}

\thanks{We acknowledge 
support from the DFG-priority program 
SPP1489 ``Algorithmic and Experimental Methods in
Algebra, Geometry, and Number Theory''.}


\begin{document}

\title[Ziegler's Multi-Reflection Arrangements are free]
{Ziegler's Multi-Reflection Arrangements are free}


\author[T. Hoge]{Torsten Hoge}
\address
{Institut f\"ur Algebra, Zahlentheorie und Diskrete Mathematik,
Fakult\"at f\"ur Mathematik und Physik,
Leibniz Universit\"at Hannover,
Welfengarten 1,
30167 Hannover, Germany}
\email{hoge@math.uni-hannover.de}

\author[G. R\"ohrle]{Gerhard R\"ohrle}
\address
{Fakult\"at f\"ur Mathematik,
Ruhr-Universit\"at Bochum,
D-44780 Bochum, Germany}
\email{gerhard.roehrle@rub.de}


\allowdisplaybreaks

\begin{abstract}
In 1989, Ziegler introduced 
the concept of a multi-arrangement.
One natural example is the reflection 
arrangement of a unitary
reflection group with multiplicity 
given by the number of reflections
associated with each hyperplane.
For all but three  irreducible groups, 
Ziegler showed that each such 
multi-reflection arrangement
is free.
We complete Ziegler's example 
by confirming these outstanding cases.
\end{abstract}

\maketitle



\section{Introduction}

In his seminal work \cite{ziegler:multiarrangements}, Ziegler
introduced the concept of a multi-arrangement
generalizing the notion of a hyperplane arrangement.
A natural example of such 
a multi-arrangement is 
the reflection arrangement of an irreducible unitary
reflection group with multiplicity 
given by the number of reflections
associated with each hyperplane.
Ziegler showed in \cite{ziegler:multiarrangements}
that each such 
multi-reflection arrangement is free with the 
possible exception of 
just three instances.
In this short note we revisit Ziegler's example 
and show by computational means that 
these remaining cases are also free 
in Theorem \ref{thm:ziegler}.

Ever since Ziegler's introduction of multi-arrangements, 
the subject flourished. 
In particular, the question of freeness of 
multi-arrangements 
is a very active field of research, 
e.g.\ see the recent work 
\cite{abeteraowakefield:euler} 
and \cite{yoshinaga:free} and the references therein.

\subsection{Multi-Arrangements}
Let $\BBK$ be a field and let $V = \BBK^\ell$.
Let $\CA = (\CA, V)$ be a central $\ell$-arrangement of hyperplanes 
in $V$.
A \emph{multi-arrangement}  is a pair
$(\CA, \nu)$ consisting of a hyperplane arrangement $\CA$ and a 
\emph{multiplicity} function
$\nu : \CA \to \BBZ_{\ge 0}$ associating 
to each hyperplane $H$ in $\CA$ a non-negative integer $\nu(H)$.
Alternately, the multi-arrangement $(\CA, \nu)$ can also be thought of as
the multi-set of hyperplanes
\[
(\CA, \nu) = \{H^{\nu(H)} \mid H \in \CA\}.
\]

The order of $\CA$ is the cardinality $|\CA|$ of the set $\CA$ and
the \emph{order} of the multi-arrangement $(\CA, \nu)$ 
is the cardinality 
of the multi-set $(\CA, \nu)$, we write 
$|\nu| := |(\CA, \nu)| = \sum_{H \in \CA} \nu(H)$.
For a multi-arrangement $(\CA, \nu)$, the underlying 
arrangement $\CA$ is sometimes called the associated 
\emph{simple} arrangement, and so $(\CA, \nu)$ itself is  
simple if and only if $\nu(H) = 1$ for each $H \in \CA$. 

\subsection{Freeness of Arrangements and Multi-Arrangements}

Let $S = S(V^*)$ be the symmetric algebra of the dual space $V^*$ of $V$.
If $x_1, \ldots , x_\ell$ is a basis of $V^*$, then we identify $S$ with 
the polynomial ring $\BBK[x_1, \ldots , x_\ell]$.
Letting $S_p$ denote the $\BBK$-subspace of $S$
consisting of the homogeneous polynomials of degree $p$ (along with $0$),
$S$ is naturally $\BBZ$-graded: $S = \oplus_{p \in \BBZ}S_p$, where
$S_p = 0$ in case $p < 0$.

Let $\Der(S)$ be the $S$-module of algebraic $\BBK$-derivations of $S$.
For $i = 1, \ldots, \ell$, 
let $D_i := \partial/\partial x_i$ be the usual derivation of $S$.
Then $D_1, \ldots, D_\ell$ is an $S$-basis of $\Der(S)$.
We say that $\theta \in \Der(S)$ is 
\emph{homogeneous of polynomial degree p}
provided 
$\theta = \sum_{i=1}^\ell f_i D_i$, 
where $f_i \in S_p$ for each $1 \le i \le \ell$.
In this case we write $\pdeg \theta = p$.
Let $\Der(S)_p$ be the $\BBK$-subspace of $\Der(S)$ consisting 
of all homogeneous derivations of polynomial degree $p$
(along with $0$).
So $\Der(S)$ is a graded $S$-module:
$\Der(S) = \oplus_{p\in \BBZ} \Der(S)_p$.

Let $\CA$ be an arrangement in $V$. 
Then for $H \in \CA$ we fix $\alpha_H \in V^*$ with
$H = \ker(\alpha_H)$.
The \emph{defining polynomial} $Q(\CA)$ of $\CA$ is given by 
$Q(\CA) := \prod_{H \in \CA} \alpha_H \in S$.

The \emph{module of $\CA$-derivations} of $\CA$ is 
defined by 
\[
D(\CA) := \{\theta \in \Der(S) \mid \theta(\alpha_H) \in \alpha_H S
\text{ for each } H \in \CA \} .
\]
We say that $\CA$ is \emph{free} if the module of $\CA$-derivations
$D(\CA)$ is a free $S$-module.

With the $\BBZ$-grading of $\Der(S)$, the module of $\CA$-derivations
becomes a graded $S$-module $D(\CA) = \oplus_{p\in \BBZ} D(\CA)_p$,
where $D(\CA)_p = D(\CA) \cap \Der(S)_p$, 
\cite[Prop.\ 4.10]{orlikterao:arrangements}.
If $\CA$ is a free arrangement, then the $S$-module 
$D(\CA)$ admits a basis of $\ell$ homogeneous derivations, 
say $\theta_1, \ldots, \theta_\ell$, \cite[Prop.\ 4.18]{orlikterao:arrangements}.
While the $\theta_i$'s are not unique, their polynomial 
degrees $\pdeg \theta_i$ 
are unique (up to ordering). This multiset is the set of 
\emph{exponents} of the free arrangement $\CA$
and is denoted by $\exp \CA$.

Following Ziegler \cite{ziegler:multiarrangements},
we extend this notion of freeness to multi-arrangements.
The \emph{defining polynomial} $Q(\CA, \nu)$ 
of the multi-arrangement $(\CA, \nu)$ is given by 
\[
Q(\CA, \nu) := \prod_{H \in \CA} \alpha_H^{\nu(H)},
\] 
a polynomial of degree $|\nu|$ in $S$.

The \emph{module of $\CA$-derivations} of $(\CA, \nu)$ is 
defined by 
\[
D(\CA, \nu) := \{\theta \in \Der(S) \mid \theta(\alpha_H) \in \alpha_H^{\nu(H)} S 
\text{ for each } H \in \CA\}.
\]
We say that $(\CA, \nu)$ is \emph{free} if 
$D(\CA, \nu)$ is a free $S$-module, \cite[Def.\ 6]{ziegler:multiarrangements}.

As in the case of simple arrangements,
if $(\CA, \nu)$ is free, there is a 
homogeneous basis $\theta_1, \ldots, \theta_\ell$ of $D(\CA, \nu)$.
The multi-set of the unique polynomial degrees $\pdeg \theta_i$ 
are the \emph{multi-exponents} of the free multi-arrangement $(\CA, \nu)$
and is denoted by $\exp (\CA, \nu)$.
It follows from Ziegler's analogue of Saito's criterion 
\cite[Thm.\ 8]{ziegler:multiarrangements} that 
$\sum \pdeg \theta_i = \deg Q(\CA, \nu) = |\nu|$.

As is the case for simple arrangements, if $\ell$ is at most $2$,
then $(\CA, \nu)$ is free, 
\cite[Cor.\ 7]{ziegler:multiarrangements}.

\section{Ziegler's Multi-Arrangement for Unitary Reflection Groups}

Now let $\BBK = \BBC$, the complex numbers. 
Suppose that $W$ is a finite, unitary
reflection group acting on the complex 
vector space $V$.
Let $\CA(W) = (\CA(W),V)$ be the associated 
hyperplane arrangement of $W$, the
\emph{reflection arrangement} of $W$.
For $w \in W$, we write 
$\Fix(w) :=\{ v\in V \mid w v = v\}$ for 
the fixed point subspace of $w$.
We use the classification and 
labeling of the irreducible 
unitary reflection groups due to
Shephard and Todd, \cite{shephardtodd}. 

Ziegler defined the multi-arrangement $(\CA(W), \varrho)$ of $W$, 
with the \emph{reflection multiplicity} $\varrho$, i.e. 
\[
\varrho(H) := |\{ w \in W \mid \Fix(w) = H\}|
\] 
is the number of
pseudo-reflections having $H$ as fixed point hyperplane.
So  $|\varrho|$ is the number of reflections in $W$ and 
the defining polynomial $Q(\CA(W), \varrho)$ of 
$(\CA(W), \varrho)$ is 
the determinant of the Jacobian of a fixed set of basic invariants of $W$, cf.\ 
\cite[Thm.\ 6.42]{orlikterao:arrangements}.

Our aim is to complete the proof of the following 

\begin{theorem}
\label{thm:ziegler}
For $W$ a  finite, unitary reflection group,  
the multi-arrangement  $(\CA(W), \varrho)$ of $W$ is free.
\end{theorem}

\begin{proof}
A product of multi-arrangements is free if and only if each factor is free:
using \cite[Lem.\ 1.3]{abeteraowakefield:euler}, 
the proof of \cite[Thm.\ 4.28]{orlikterao:arrangements}
readily extends to multi-arrangements, thanks to 
Ziegler's analogue of Saito's criterion 
\cite[Thm.\ 8]{ziegler:multiarrangements}.
Thus we may assume that $W$ is irreducible.

All but three cases were proved by Ziegler in 
\cite{ziegler:multiarrangements}.
If $W$ is generated by pseudo-reflections of order 2, e.g. if $W$ is a Coxeter group,
then  $\varrho \equiv 1$.
Thus in these instances 
$(\CA(W), \varrho) = \CA(W)$ is simple. In these cases
$\CA(W)$ is known to be free, thanks to 
Terao's work, \cite{terao:freeI}.
If $W$ is cyclic or of rank $2$, then $(\CA(W), \varrho)$ is free,
by \cite[Cor.\ 7]{ziegler:multiarrangements}.
Also for the monomial groups
$W = G(r,p,\ell)$,  Ziegler showed that 
$(\CA(W), \varrho)$ is free, \cite[Ex.\ 15]{ziegler:multiarrangements}.

So the question about freeness of $(\CA(W), \varrho)$ is only outstanding
for the three exceptional groups $W = G_{25}, G_{26}$ and $G_{32}$.
Both $G_{25}$ and $G_{32}$ are generated by pseudo-reflections of order $3$,
while $G_{26}$ admits $9$ pseudo-reflections of order $2$ and $24$ of order $3$.
Therefore, we have 
$|\CA(W)| = 12, 21, 40$
and 
$|(\CA(W), \varrho)| = 24, 33, 80$, respectively.

Our proof of these remaining cases for 
Theorem \ref{thm:ziegler} is computational.
First we use the functionality for complex reflection groups 
provided by the   \CHEVIE\ package in   \GAP\ 
(and some \GAP\ code by J.~Michel)
(see \cite{gap3} and \cite{chevie})
in order to obtain explicit 
linear functionals $\alpha_H$ so that 
$H = \ker \alpha_H$ for the underlying reflection arrangement
$\CA(W)$. 
These then allow us to implement the 
$S$-module  
$D(\alpha_H, \varrho) 
:= \{\theta \in \Der(S) \mid \theta(\alpha_H) \in \alpha_H^{\varrho(H)} S \}$
associated with $\alpha_H$
in the   \Singular\ computer algebra system (cf.~\cite{singular}). 
Then the
module theoretic functionality of
  \Singular\ is used to show that the
modules of derivations in question 
\[
D(\CA(W), \varrho) = \cap_{H \in \CA(W)} D(\alpha_H, \varrho)
\]
are free.
In particular, 
for $W = G_{25}, G_{26}$ and $G_{32}$, 
the multi-exponents are
$\exp (\CA(W), \varrho) =\{8,8,8\}, \{10,10,13\}$ and $\{20,20,20,20\}$,
respectively.
As an illustration, 
we give explicit $S$-bases 
for $D(\CA(G_{25}),\varrho)$ and  
$D(\CA(G_{26}),\varrho)$
in the next section.
Not unexpectedly, they are not particularly enlightening.
\end{proof}

Note that
even though the simple arrangement $\CA(W)$ is free
\cite{terao:freeI}, for  
an arbitrary multiplicity $\nu$ of $\CA(W)$, 
the multi-arrangement $(\CA(W), \nu)$ 
need not be free in general, cf.\ 
\cite[Ex.\ 5.13]{abeteraowakefield:euler}.

While our calculations 
combined with the existing known instances determined by Ziegler
do provide a proof of Theorem \ref{thm:ziegler}, it would  
nevertheless be very desirable to have a uniform,  conceptual
proof free of case-by-case considerations and free of 
computer calculations.

\section{Defining Polynomials and Bases of $D(\CA(W),\varrho)$}

To illustrate our computations, 
we list explicit $S$-bases for $D(\CA(G_{25}),\varrho)$ and  
$D(\CA(G_{26}),\varrho)$. 
Let $x, y$, and $z$ be the indeterminates of $S$, 
$D_x = \partial/\partial x, D_y = \partial/\partial y$, 
$D_z = \partial/\partial z$, and let
$\zeta$ be a primitive $3$rd root of unity.

\begin{align*}
Q(\CA(G_{25}), \varrho) = & \  Q(G_{25})^3 = 
\big(x y z (x + y + z) (x + y + \zeta z) (x + y - (\zeta+1) z)   \\
  & (x + \zeta y + z) (x + \zeta y + \zeta z) (x + \zeta y - (\zeta+1) z) \\
  & (x - (\zeta+1) y + z) (x - (\zeta+1) y + \zeta z) (x - (\zeta+1) y - (\zeta+1) z) \big)^{3}.
\end{align*}

\[
\renewcommand{\arraystretch}{1.4}
\begin{array}{lll}
D(\CA(G_{25}),\varrho) 
& \!\!\! = S \big(& \hskip-10pt (6 x^{7} z + 42 x^{4} y^{3} z - 21 x^{4} z^{4})D_x
 + (42 x^{3} y^{4} z + 6 y^{7} z - 21 y^{4} z^{4}) D_y  \\
&  & + (14 x^{6} z^{2} + 28 x^{3} y^{3} z^{2} + 14 y^{6} z^{2} - 14 x^{3} z^{5}
      - 14 y^{3} z^{5} -  z^{8})D_z \big) \\
& \!\!\! +\ S \big( & \hskip-10pt (6 x^{7} y - 21 x^{4} y^{4} + 42 x^{4} y z^{3})D_x+ (42 x^{3} y z^{4} - 21 y^{4} z^{4} + 6 y z^{7})D_z  \\
&     & + (14 x^{6} y^{2} - 14 x^{3} y^{5} -  y^{8} + 28 x^{3} y^{2} z^{3} - 14
     y^{5} z^{3} + 14 y^{2} z^{6}) D_y \big) \\
& \!\!\! +\ S \big( & \hskip-10pt  (x^{8} + 14 x^{5} y^{3} - 14 x^{2} y^{6} + 14 x^{5} z^{3} 
- 28 x^{2} y^{3} z^{3} - 14 x^{2} z^{6})D_x \\
& & + (21 x^{4} y^{4} - 6 x y^{7} - 42 x y^{4} z^{3}) D_y  
+ (21 x^{4} z^{4} - 42 x y^{3} z^{4} - 6 x z^{7})D_z \big).
\end{array}
\]

\begin{align*}
Q(\CA(G_{26}), \varrho) = & \  
(y -  z)^{2}
 (x -  z)^{2}
 (x -  y)^{2}
 (y - \zeta z)^{2}
(x - \zeta z)^{2}
 (x - \zeta y)^{2}
(y + \left(\zeta + 1\right) z)^{2}\\
&  
 (x + \left(\zeta + 1\right) y)^{2}
(x + \left(\zeta + 1\right) z)^{2}
 x^{3} y^{3}z^{3}
 (x + y + z)^{3}
 (x + \left(-\zeta - 1\right) y + \zeta z)^{3}\\
& 
(x + y + \zeta z)^{3}
 (x + y + \left(-\zeta - 1\right) z)^{3}  
 (x + \zeta y + z)^{3}
 (x + \left(-\zeta - 1\right) y + z)^{3}\\
& 
 (x + \zeta y + \left(-\zeta - 1\right) z)^{3} 
 (x + \left(-\zeta - 1\right) y + \left(-\zeta - 1\right) z)^{3}  
 (x + \zeta y + \zeta z)^{3}.
\end{align*}

\[
\renewcommand{\arraystretch}{1.4}
\begin{array}{lll}
D(\CA(G_{26}),\varrho) 
& \!\!\! = S \big(& \hskip-10pt (
11 x^{8} y z + 7 x^{5} y^{4} z + 14 x^{2} y^{7} z + 7 x^{5} y z^{4} + 28 x^{2} y^{4} z^{4} + 14 x^{2} y z^{7})D_x \\
&  & + (14 x^{7} y^{2} z + 7 x^{4} y^{5} z + 11 x y^{8} z + 28 x^{4} y^{2} z^{4} + 7 x y^{5} z^{4} + 14 x y^{2} z^{7})D_y  \\
&  & + (14 x^{7} y z^{2} + 28 x^{4} y^{4} z^{2} + 14 x y^{7} z^{2} + 7 x^{4} y z^{5} + 7 x y^{4} z^{5} + 11 x y z^{8})D_z \big) \\
& \!\!\! +\ S \big( &  \hskip-10pt  (
x^{10} + 8 x^{7} y^{3} + 7 x^{4} y^{6} + 8 x^{7} z^{3} - 112 x^{4} y^{3} z^{3} + 7 x^{4} z^{6})D_x \\
&     & + (7 x^{6} y^{4} + 8 x^{3} y^{7} + y^{10} - 112 x^{3} y^{4} z^{3} + 8 y^{7} z^{3} + 7 y^{4} z^{6}) D_y \\
&     & + (7 x^{6} z^{4} - 112 x^{3} y^{3} z^{4} + 7 y^{6} z^{4} + 8 x^{3} z^{7} + 8 y^{3} z^{7} + z^{10})D_z \big) \\
& \!\!\! +\ S \big( & \hskip-10pt  (
-75 x^{7} y^{6} - 21 x^{4} y^{9} - 12 x^{7} y^{3} z^{3} + 588 x^{4} y^{6} z^{3} - 75 x^{7} z^{6} + 588 x^{4} y^{3} z^{6} - 21 x^{4} z^{9})D_x \\
& & + (14 x^{9} y^{4} - 70 x^{6} y^{7} - 35 x^{3} y^{10} - 5 y^{13} + 28 x^{6} y^{4} z^{3} \\ 
& & \quad + 588 x^{3} y^{7} z^{3} - 40 y^{10} z^{3} + 623 x^{3} y^{4} z^{6} - 110 y^{7} z^{6} - 21 y^{4} z^{9}) D_y  \\
& & + (14 x^{9} z^{4} + 28 x^{6} y^{3} z^{4} + 623 x^{3} y^{6} z^{4} - 21 y^{9} z^{4} - 70 x^{6} z^{7}  \\ 
& & \quad + 588 x^{3} y^{3} z^{7} - 110 y^{6} z^{7} - 35 x^{3} z^{10} - 40 y^{3} z^{10} - 5 z^{13})D_z \big).
\end{array}
\]



\bigskip

\bibliographystyle{amsalpha}

\newcommand{\etalchar}[1]{$^{#1}$}
\providecommand{\bysame}{\leavevmode\hbox to3em{\hrulefill}\thinspace}
\providecommand{\MR}{\relax\ifhmode\unskip\space\fi MR }
\providecommand{\MRhref}[2]{%
  \href{http://www.ams.org/mathscinet-getitem?mr=#1}{#2} }
\providecommand{\href}[2]{#2}


\end{document}